\documentclass[12pt]{amsart}
\usepackage{amsthm,amssymb,amsmath,amscd}

\newtheorem{theorem}{Theorem}[section]
\newtheorem{corollary}[theorem]{Corollary}
\newtheorem{lemma}[theorem]{Lemma}
\newtheorem{proposition}[theorem]{Proposition}
\theoremstyle{definition}
\newtheorem{definition}[theorem]{Definition}
\newtheorem{example}[theorem]{Example}
\newtheorem{remark}[theorem]{Remark}

\title{Antichains and counterpoint dichotomies}

\author{Octavio A. Agust\'{\i}n-Aquino}
\thanks{The author was supported by a CONACyT grant, with fellow number 207309.}
\address{Universidad de la Ca\~nada}
\email{octavioalberto@unca.edu.mx}

\begin{document}
\subjclass[2000]{00A65, 05A20, 05E18, 05D05, 11Z05}
\keywords{complement-free antichain, counterpoint dichotomy}
\begin{abstract}
We construct a special type of antichain (i. e., a family of subsets
of a set, such that no subset is contained in another) using group-theoretical
considerations, and obtain an upper bound on the cardinality of such an antichain.
We apply the result to bound the number of strong counterpoint
dichotomies up to affine isomorphisms.
\end{abstract}
\maketitle

\section{Introduction}

Sperner systems, antichains or clutters (families of subsets of an $n$-set such that
no one is contained in another) have been important combinatorial objects since
the foundational theorem of Sperner was published, which states that they cannot
contain more than $\binom{n}{\lfloor n/2\rfloor}$ elements. Concerning their
size, it is possible to establish upper bounds on their cardinality using
elegant elementary methods and any additional properties they may have. Sometimes it is even 
possible to exhibit maximal antichains that attain such bounds.

In particular, some antichains can be studied from the standpoint of group theory as in \cite{EFK92},
to obtain bounds (in the spirit of the Erd\H{o}s-Ko-Rado theorem) that rely on the
isotropy groups of the members of the antichain. However, particular
types of antichains that are of interest for mathematical musicology 
(namely, pairwise non-isomorphic counterpoint dichotomies) are not encompassed
in such an approach, for the isotropy group of their members is trivial.

In Section 2 we define such antichains in a slightly more general
setting than the original formulation given in \cite[Part VII]{gM02}, and we obtain a somewhat crude upper bound for
their cardinality. Then, in Section 3, we apply the result
to bound the number of non-isomorphic counterpoint dichotomies.

\section{Strong antichains}

Let $G$ be a subgroup of the symmetric group of $S$. The group $G$
acts obviously on $S$; it also acts on the powerset $2^{S}$ of $S$, with an action defined by
\begin{align*}
 \cdot:G\times 2^{S}&\to 2^{S},\\
 (g,A) &\mapsto gA:=\{ga:a\in A\}.
\end{align*}

Two sets $A,B\in 2^{S}$ are said to be \emph{isomorphic} if there exists $g\in G$
such that $A=gB$. In this paper we will always consider a set $S$ of even cardinality and the subset
\[
 \binom{S}{|S|/2}:=\{A\subseteq S:|A|=|S|/2\},
\]
of its power set, whose elements will be called \emph{dichotomies}.

\begin{definition}
The dichotomy $D\in \binom{S}{|S|/2}$ is:
\begin{enumerate}
\item $G$-\emph{rigid} when $gD=D$ for $g\in G$ implies $g=e$,
\item $G$-\emph{autocomplementary} if there exists a $p\in G$ (called \emph{autocomplementarity function})
such that
\[
 pD=\complement D:=S\setminus D,
\]
\item $G$-\emph{strong} if it is $G$-rigid and $G$-autocomplementary.
\end{enumerate}
\end{definition}

\begin{remark}\hspace{1em}
\begin{enumerate}
\item If a dichotomy $D\in \binom{S}{|S|/2}$ is $G$-strong and $p\in G$ is such that $pD=\complement D$,
then $p$ is unique. We call it the \emph{polarity} of $D$. For any $g\in G$, the set $gD$ is also $G$-strong (with
polarity $gpg^{-1}$).
\item Since $S$ is finite, any polarity is an involution that does not fix points of $S$,
i. e., it is an involutive derangement.
\end{enumerate}
\end{remark}

\begin{definition}
An involutive derangement $q\in G$ is called a \emph{quasipolarity}.
\end{definition}

Since any quasipolarity $q$ decomposes only in $2$-cycles, we can construct a dichotomy
$U_{q}\in \binom{S}{|S|/2}$ that contains exactly one element from each cycle.
It is clear that $U_{q}$ is autocomplementary and $qU_{q}=\complement U_{q}$.

Consider now the set of dichotomies whose autocomplementarity function is $q$, which is
\begin{align*}
M_{q} &=\left\{D\in\binom{S}{|S|/2}:qD=S\setminus D\right\}\\
&=\{A\cup B: A\subseteq U_{q}, B=(S\setminus U_{q})\setminus qA\},
\end{align*}
where the last equality follows from
\[
 qA = (S\setminus U_{q})\setminus A\quad\text{and}\quad q((S\setminus U_{q})\setminus qA) = U_{q}\setminus A.
\]

We see that $|M_{q}| = 2^{|U_{q}|} = 2^{|S|/2}$.

\begin{lemma}\label{L:AccPol1}
If $g\in G_{q}$, then $gM_{q}=M_{q}$.
\end{lemma}
\begin{proof}
Since $g\in G_{q}$ is equivalent to $gq = qg$, we have
\[
q(gD)= g(qD) = g(S\setminus D) = S\setminus gD,
\]
which means that $gD\in M_{q}$, for $|gD|=S/2$.
\end{proof}

\begin{lemma}\label{L:AccPol2}
If $h\in G$, then $M_{hqh^{-1}}=hM_{q}$.
\end{lemma}
\begin{proof}
First, we have
\[
hM_{q} = \{hD:qD=S\setminus D\}.
\]

Now take any $hD\in hM_{q}$, then
\[
hqh^{-1}hD = hqD = h(S\setminus D) = S\setminus hD,
\]
thus $hD\in M_{hqh^{-1}}$ and $hM_{q}\subseteq M_{hqh^{-1}}$. Since both sets in the inclusion have the
same cardinality, the lemma follows.
\end{proof}


Let $\mathcal{T}$ be a traversal of $\binom{S}{|S|/2}/G$ (i. e., a set consisting of exactly one
element from each $G$-orbit on $\binom{S}{|S|/2}/G$) and $\mathcal{R}_{G}\subseteq \mathcal{T}$
a subset such that all of its members are $G$-strong. For any $D_{1},D_{2}\in \mathcal{T}$, it is impossible
that $D_{1}\subseteq D_{2}$, for then $D_{1}=D_{2}$ and they would represent the same orbit.
Thus $\mathcal{T}$ and $\mathcal{R}_{G}$ are antichains. In particular, $\mathcal{R}_{G}$ is a \emph{complement-free} antichain, since
$D\in \mathcal{R}_{G}$ implies that $\complement D \notin \mathcal{R}_{G}$. Indeed,
if $D_{1}\cap D_{2} = \emptyset$, it would imply that $D_{1}=\complement D_{2}=pD_{2}$ for some $p\in G$,
hence $D_{1}$ would be in the orbit of $D_{2}$.

\begin{definition}
Let $G$ be a subgroup of the symmetric group of a finite set $S$, with $|S|$ even.
A $G$-\emph{strong antichain} is a subset $\mathcal{R}_{G}$ of a traversal $\mathcal{T}$ of the orbit set $\binom{S}{|S|/2}/G$,
such that all its members are $G$-strong.
\end{definition}

It is obvious that the cardinality of a $G$-strong antichain $\mathcal{R}_{G}$
is not greater than the number of $(|S|/2)$-subsets of $S$:
\[
 |\mathcal{R}_{G}|\leq \binom{|S|}{|S|/2}.
\]

Being $\mathcal{R}_{G}$ a complement-free antichain, a theorem by Purdy \cite[p. 139]{iA02} tells us that
\[
 |\mathcal{R}_{G}|\leq \binom{|S|}{|S|/2-1},
\]
and using the Erd\H{o}s-Ko-Rado theorem \cite[Theorem 1]{EKR61}, we can improve this slightly:
\[
 |\mathcal{R}_{G}|\leq \binom{|S|-1}{|S|/2-1}.
\]

These upper bounds, however, do not make full use of the structure of $\mathcal{R}_{G}$ derived from
the action of $G$ on $S$. In order to exploit it, note first that $G$-strong dichotomies have orbits
of maximum cardinality, namely $|G|$.

Let $G_{q}$ be the isotropy group of a quasipolarity $q\in G$ under the conjugation action. The number of non-isomorphic 
$G_{q}$-strong dichotomies for a given quasipolarity $q$ is bounded by
\begin{equation}\label{E:CotaBurdap}
|\mathcal{R}_{G_{q}}| \leq \frac{|M_{q}|}{|G_{q}|} = \frac{2^{|S|/2}}{|G_{q}|},
\end{equation}
since, by Lemmas \ref{L:AccPol1} and \ref{L:AccPol2}, the set $M_{q}$ of dichotomies whose autocomplementary
function is $q$ coincide with $M_{q'}$, for any conjugate $q'$ of $q$.

Every $G$-strong dichotomy is a $G_{q}$-strong dichotomy because $G_{q}$ is
a subgroup of $G$, hence $|\mathcal{R}_{G}|\leq \sum_{[q]\in\mathcal{Q}}|\mathcal{R}_{G_{q}}|$,
where $\mathcal{Q}$ is the set of conjugacy classes of the quasipolarities of $G$.
Now summing \eqref{E:CotaBurdap} over
$\mathcal{Q}$, we prove the following
result.

\begin{theorem}\label{T:Cota}
Let $G$ be a subgroup of the symmetric group of the set $S$
(where $|S|$ is an even number) and $\mathcal{R}_{G}$ be a $G$-strong antichain. Then
\begin{equation}\label{E:CotaBurda}
|\mathcal{R}_{G}| \leq \sum_{[q]\in \mathcal{Q}}\frac{2^{|S|/2}}{|G_{q}|},
\end{equation}
where $\mathcal{Q}$ is the set of conjugacy classes of quasipolarities of $G$.
\end{theorem}

\begin{remark}
Observe that if we have better estimations of the
number of quasipolarities of $G$ and of the sizes of their isotropy groups,
we can improve the bound of Theorem \ref{T:Cota}. But even with those
refinements, \eqref{E:CotaBurda} might be far from optimal.
\end{remark}

\begin{example} If $|S|=2k>2$ and $G=\mathrm{Sym}(S)$ (the symmetric group), there are no $\mathrm{Sym}(S)$-strong dichotomies.
Considering that the number of conjugacy classes of quasipolarities in this case is
$1$ and the cardinality of the isotropy group of any quasipolarity $q$ satisfies
\[
|\mathrm{Sym}(S)_{q}|=2^{k}\cdot k!,
\]
we conclude using \eqref{E:CotaBurda} that
\[
|\mathcal{R}_{\mathrm{Sym}(S)}|\leq \frac{2^{k}}{2^{k}\cdot k!} = \frac{1}{k!},
\]
which is a nice confirmation.
\end{example}

\begin{example}
If $|S|=2k$ and $G=D_{2k}$ is the dihedral group of order $4k$, there are $k+1$ quasipolarities: $k$
reflections and one rotation. The reflections constitute a whole conjugacy class, each having isotropy groups of cardinality $2$, whilst
the rotation has $D_{2k}$ as its isotropy group. This, together with \eqref{E:CotaBurda},
gives us
\[
 |\mathcal{R}_{D_{2k}}| \leq \frac{2^{k}}{2}+\frac{2^k}{4k} = 2^{k-1}+\frac{2^{k-2}}{k}.
\]
\end{example}

\section{The case of counterpoint dichotomies}

As we stated in the introduction, strong $G$-antichains are of great interest for the mathematical
theory of counterpoint, for they represent alternative ways of understanding consonances and
dissonances (see \cite{oA09} and \cite{gM02}, for example). In that context, we take $S=\mathbb{Z}_{2k}$ and
$G$ is the semidirect product
\[
 G = \overrightarrow{GL}(\mathbb{Z}_{2k}):=\mathbb{Z}_{2k} \ltimes \mathbb{Z}_{2k}^{\times}
\]
where $\mathbb{Z}_{2k}^{\times}$ is the set of units of
$\mathbb{Z}_{2k}$. The elements of $G$ can be seen as
the \emph{affine} functions
\begin{align*}
 e^{u}v:\mathbb{Z}_{2k}&\to \mathbb{Z}_{2k},\\
 x&\mapsto vx+u,
\end{align*}
the element $u$ is the \emph{affine part} and $v$ is the \emph{linear part}. Note that
\[
 |\overrightarrow{GL}(\mathbb{Z}_{2k})| = 2k\varphi(2k)
\]
where $\varphi$ is the Euler totient function.

Our aim now is to characterize the quasipolarities
of $\overrightarrow{GL}(\mathbb{Z}_{2k})$ in terms of the
involutions of $\mathbb{Z}_{2k}^{\times}$. We will use this
information to obtain a bound on the number of strong $\overrightarrow{GL}(\mathbb{Z}_{2k})$-dichotomies
(which we will call \emph{counterpoint dichotomies}) up to isomorphism.

In order to state the characterization, for $v\in \mathbb{Z}_{2k}^{\times}$ define
\begin{align*}
 \sigma_{2k}(v) &:= \gcd(\nu-1,2k),\\
 \tau_{2k}(v) &:= \gcd(\nu+1,2k),
\end{align*}
where $\nu$ is taken as the minimum non-negative element in the residue class of $v$.

\begin{theorem}\label{T:CaracQ}
The function $e^{u} v \in \overrightarrow{GL}(\mathbb{Z}_{2k})$
is a quasipolarity if, and only if:
\begin{enumerate}
\item $v$ is an involution,
\item $2\frac{2k}{\tau_{2k}(v)} = \sigma_{2k}(v)$ and
\item $u = \sigma_{2k}(v)q+\frac{2k}{\tau_{2k}(v)}$ for some
integer $q$.
\end{enumerate}

If $k$ is odd, the second condition can be omitted.
\end{theorem}
Before proving the theorem, let us make a series of remarks. Let
$g=e^{u}v \in \overrightarrow{GL}(\mathbb{Z}_{2k})$. The function
$g$ is an affine involution if, and only if,
\begin{align}
v^{2} &= 1,\label{E:Cond1}\\
u(v+1)&=0.\label{E:Cond2}
\end{align}

For $g$ to be a derangement, it is necessary and sufficient that the
equation
\begin{equation}\label{E:Cond3}
 (v-1)x = -u
\end{equation}
has no solutions. It is easy to see that the only values of $u$ that satisfy \eqref{E:Cond2} are the multiples of
\[
 u_{0} = \frac{2k}{\tau_{2k}(v)}.
\]

\begin{lemma}\label{L:Div}
Let $v\in \mathbb{Z}_{2k}^{\times}$ be an involution. Then
\[
 u_{0}=\frac{2k}{\tau_{2k}(v)}\quad \text{divides}\quad \sigma_{2k}(v).
\]
\end{lemma}
\begin{proof}
By Bezout's identity, there are integers $a,b,c,d$ such that
\[
 \sigma_{2k}(v) = a(\nu-1)+2kb, \quad \tau_{2k}(v) = c(\nu+1)+2kd.
\]

We have the product
\[
 \sigma_{2k}(v)\tau_{2k}(v) = ac(\nu^{2}-1)+2k(\nu(bc+ad)+bc-ad+2kbd).
\]

Since $v^{2}-1=0$, there exists an integer $m$ such that $\nu^{2}-1=2km$, hence
\[
 \sigma_{2k}(v)\tau_{2k}(v) = 2k(acm+\nu(bc+ad)+bc-ad+2kbd)
\]
and the result is immediate.
\end{proof}

Equation \eqref{E:Cond3} does not hold if, and only if,
$\sigma_{2k}(v)=\gcd(\nu-1,n)$ does not divide $-u$ (or, equivalently,
it does not divide $u$).

By the division algorithm, $u$ must be a multiple of $u_{0}$ of the form
\begin{equation}\label{E:PQND}
 u=\sigma_{2k}(v) q + r, \quad 0<r<\sigma_{2k}(v), q\in \mathbb{Z}
\end{equation}

From this and Lemma \ref{L:Div} we deduce that $r$
must also be a multiple of $u_{0}$. It is clear that $u_{0}\leq \sigma_{2k}(v)$.

\begin{proposition}\label{P:Nohay}
If $u_{0}=\frac{2k}{\tau_{2k}(v)}=\sigma_{2k}(v)$, there
are no quasipolarities with $v$ as its linear part.
\end{proposition}
\begin{proof}
For if $u_{0}=\sigma_{2k}(v)$, then $\sigma_{2k}(v)$ divides $u$.
\end{proof}

\begin{proof}[Proof of Theorem \ref{T:CaracQ}]
Suppose $e^{u}v$ is a quasipolarity. Then equations \eqref{E:Cond1}, \eqref{E:Cond2} hold and
\eqref{E:Cond3} has no solutions. Consequently, $u$ is a multiple of $u_{0}$. The integer $u$ guarantees that \eqref{E:Cond3}
has no solutions if, and only if, 
\[
 u = \sigma_{2k}(v) q + r
\]
for some integer $q$ and $0<r<\sigma_{2k}(v)$. This happen if, and only if, $u_{0}<\sigma_{2k}(v)$. The
inequality is sufficient because $u_{0}$ divides $\sigma_{2k}(v)$ and Proposition \ref{P:Nohay}; it is necessary because $u_{0}$ divides $r$ and
$u_{0}\leq r <\sigma_{2k}(v)$.

For the converse, let $v$ be an involution, let $u_{0}=\frac{2k}{\tau_{2k}(v)}$, and let $u$ be given by \eqref{E:PQND} with $r=u_{0}$. Thus $u$ is
a multiple of $u_{0}$, hence it is a solution of \eqref{E:Cond2}. It also ensures that \eqref{E:Cond3}
has no solutions since $u_{0}< \sigma_{2k}(\nu)$.

In particular, $u_{0}<\sigma_{2k}(v)$ is always true when $k$ is odd. Keeping in mind that $\nu$ odd because $v$ is an involution,
let $\lambda$ and $\mu$ be such that
\begin{equation}\label{E:LambdaMu}
 \nu-1 = 2\lambda,\quad \nu+1=2\mu,
\end{equation}
note that $\lambda$ and $\mu$ are coprime because
\[
 \mu-\lambda = \frac{\nu+1}{2}-\frac{\nu-1}{2} = 1.
\]

We have
\[
2km = \nu^{2}-1 = 4\lambda\mu
\]
thus
\[
 km = 2\lambda\mu
\]
and therefore $2$ divides $m$. It follows that $\lambda\mu$ is a
multiple of $k$. Using this we can show that
$\frac{k}{\gcd(\mu,k)}=\gcd(\lambda,k)$, 
which follows from the examination of the product
\begin{align*}
 \gcd(\lambda,k)\gcd(\mu,k) &= \gcd(\lambda\mu,k)\\
 &= \gcd\left(k\frac{m}{2},k\right)\\
 &= k\gcd\left(\frac{m}{2},1\right) =k.
\end{align*}

Now
\begin{equation}\label{E:Imp}
 \sigma_{2k}(v) = 2\gcd(\lambda,k) = 2\frac{k}{\gcd(\mu,k)} = 2\frac{2k}{\tau_{2k}(v)}
 = 2u_{0}.
\end{equation}

It remains to prove that $\sigma_{2k}(v) = 2u_{0}$ when $k$ is even and $\frac{2k}{\tau_{2k}(v)} < \sigma_{2k}(v)$. The number
$\frac{k}{\gcd(\lambda,k)\gcd(\mu,k)}$ is a non-zero integer and, by hypothesis,
\[
 \frac{k}{\gcd(\lambda,k)\gcd(\mu,k)} = \frac{4k}{(2\gcd(\lambda,k))(2\gcd(\mu,k))} = 2\frac{2k}{\sigma_{2k}(v)\tau_{2k}(v)} < 2
\]
which implies that $\frac{k}{\gcd(\lambda,k)\gcd(\mu,k)}=1$ and \eqref{E:Imp} is true again.
\end{proof}

\begin{proposition}\label{P:Carac}
The isotropy group of an affine quasipolarity $q=e^{u} v \in \overrightarrow{GL}(\mathbb{Z}_{2k})$ under
the conjugation action has cardinality
\[
 |\overrightarrow{GL}(\mathbb{Z}_{2k})_{g}|=\sigma_{2k}(v)\varphi(2k).
\]

Equivalently, there are $\frac{2k}{\sigma_{2k}(v)}$ elements in the orbit of $g$ under the conjugation action.
\end{proposition}
\begin{proof}
Let $h=e^{t} s\in \overrightarrow{GL}(\mathbb{Z}_{2k})$. We have
\begin{equation}\label{E:Acc}
 g'=hg h^{-1} = (e^{t} s)(e^{u} v) (e^{-s^{-1}t} s^{-1})
 = e^{t(1-v)+su} v.
\end{equation}

The members of the stabilizer of $g$ are those whose parameters satisfy
\[
t(1-v)=u(1-s),
\]
this equation has a solution for each $s\in GL(\mathbb{Z}_{2k})$ since
$\sigma_{2k}(v) = \gcd(1-\nu,2k)$ divides $u(1-s)$ being $1-s$ is even.

In particular, the equation above has exactly $\sigma_{2k}(v)$ different solutions. Since
this occurs for each $s\in GL(\mathbb{Z}_{2k})$, the cardinality of the stabilizer
of $g$ is
\[
 |\overrightarrow{GL}(\mathbb{Z}_{2k})_{g}| = \sigma_{2k}(v)\varphi(2k).
\]

This means that each orbit has
\[
 \frac{|\overrightarrow{GL}(\mathbb{Z}_{2k})|}{\sigma_{2k}(v)\varphi(2k)} = \frac{2k}{\sigma_{2k}(v)}
\]
elements.
\end{proof}

\begin{corollary}
The group $\overrightarrow{GL}(\mathbb{Z}_{2k})$ acts transitively on the set of all
affine quasipolarities with fixed linear part.
\end{corollary}

For a $\overrightarrow{GL}(\mathbb{Z}_{2k})$-strong antichain, we have by Proposition
\ref{P:Carac}
\[
|\overrightarrow{GL}(\mathbb{Z}_{2k})_{g}| = \sigma_{2k}(v)\varphi(2k) \geq 2\varphi(2k)
\]
using the involution $v=-1$. The number of conjugacy classes of quasipolarities is bounded by $\varphi(2k)$,
thus Theorem \ref{T:Cota} gives us
\[
 |\mathcal{R}_{\overrightarrow{GL}(\mathbb{Z}_{2k})}| \leq \varphi(2k)\frac{2^{k}}{2\varphi(2k)} = 2^{k-1}.
\]

This estimate can be improved further with a better bound on the number
of involutions of $GL(\mathbb{Z}_{2k})$.


\section*{Acknowledgments}

The author wishes to thank the referee of this paper for indicating many errors and providing suggestions which led to
major improvements with respect to its first version; in particular, for contributing Lemmas \ref{L:AccPol1} and \ref{L:AccPol2} which filled
an important gap. He is also grateful to Guerino Mazzola and Jeffrey Dinitz, who provided both advice and encouragement.
Last, but not least, he expresses his profound gratitude to Teresa Mary Joselyn for checking the English of this paper.

\end{document}